\documentclass[10pt]{amsart}
\usepackage{amssymb}
\usepackage{epsfig}
\usepackage{url}
\usepackage{setspace}
\usepackage{color}
\theoremstyle{plain}

\newtheorem{thm}{Theorem}[section]
\newtheorem{cor}[thm]{Corollary}
\newtheorem{lem}[thm]{Lemma}

\newtheorem{rem}[thm]{Remark}
\newtheorem{ques}[thm]{Question}
\newtheorem{conj}[thm]{Conjecture}
\newtheorem{prob}[thm]{Problem}

\def\cal{\mathcal}
\def\bbb{\mathbb}
\def\op{\operatorname}
\renewcommand{\phi}{\varphi}
\newcommand{\R}{\bbb{R}}
\newcommand{\N}{\bbb{N}}
\newcommand{\Z}{\bbb{Z}}

\newcommand{\C}{\bbb{C}}

\makeatletter
\let\@@pmod\pmod
\DeclareRobustCommand{\pmod}{\@ifstar\@pmods\@@pmod}
\def\@pmods#1{\mkern4mu({\operator@font mod}\mkern 6mu#1)}
\makeatother

\begin{document}
\title[Signs behaviour of sums of weighted numbers of partitions]{Signs behaviour of sums of weighted numbers of partitions}
\author{Filip Gawron and Maciej Ulas}

 \subjclass[2010]{}
\thanks{Research of the authors was supported by a grant of the National Science Centre (NCN), Poland, no. UMO-2019/34/E/ST1/00094}

\begin{abstract}
Let $A$ be a subset of positive integers. By $A$-partition of $n$ we understand the representation of $n$ as a sum of elements from the set $A$. For given $i, n\in\N$, by $c_{A}(i,n)$ we denote the number of $A$-partitions of $n$ with exactly $i$ parts. In the paper we obtain several result concerning sign behaviour of the sequence $S_{A,k}(n)=\sum_{i=0}^{n}(-1)^{i}i^{k}c_{A}(i,n)$, where $k\in\N$ is fixed. In particular, we prove that for a broad class $\cal{A}$ of subsets of $\N_{+}$ we have that for each $A\in \cal{A}$ we have $(-1)^{n}S_{A,k}(n)\geq 0$ for each $n, k\in\N$.
\end{abstract}

\maketitle

\section{Introduction}\label{sec1}

Let $\N$ and $\N_{+}$ denotes the set of non-negative integers and the set of positive integers, respectively. Let $A\subset\N_{+}$ be given and take $n\in\N$. By an $A$-partition $\lambda=(\lambda_{1},\ldots, \lambda_{k})$  with parts in $A$, of a non-negative integer $n$, we mean a representation of $n$ in the form
$$
n=\lambda_{1}+\ldots+\lambda_{k},
$$
where $\lambda_{i}\in A$. The representations of $n$ differing only in order of the terms are counted as one. By $p_{A}(n)$ we denote the number of all $A$-partitions of $n$. We also use the standard convention $p_{A}(0)=1$. The famous partition function $p(n):=p_{\N_{+}}(n)$ was introduced by Euler. However, a true explosion of works dedicated to the sequence $(p(n))_{n\in\N}$ was initiated by Ramanujan and his collaborations with Hardy. The literature concerning various arithmetic properties of $p(n)$ and others partition functions associated with specific subsets of $\N_{+}$ is enormous. The standard reference covering various aspects of the theory partitions is the monograph of Andrews \cite{And0}.

The common starting point in investigations of the sequence $(p_{A}(n))_{n\in\N}$ is the ordinary  generating function
$$
\sum_{n=0}^{\infty}p_{A}(n)x^{n}=\prod_{a\in A}\frac{1}{1-x^{a}}.
$$
The above identity can be seen twofold: as a identity in the formal power series ring and also as a equality of two analytic functions (in the disc $\{z\in\C:\;|z|<1\}$) represented by corresponding sides.

For given $A\subset \N_{+}$ let us write
$$
F_{A}(t,x)=\prod_{a\in A}\frac{1}{1-tx^{a}}=\sum_{n=0}^{\infty}f_{A,n}(t)x^{n}
$$
and observe that the polynomial $f_{A,n}(t)$ is a natural generalization of the number $p_{A}(n)=f_{A,n}(1)$ which counts the number of $A$-partitions of the number $n$.

Let us write
$$
f_{A,n}(t)=\sum_{i=0}^{n}c_{A}(i,n)t^{i},
$$
and note that $c_{A}(i,n)$ is the number of $A$-partitions of $n$ with exactly $i$ parts. In particular $f_{A,n}(1)=\sum_{i=0}^{n}c_{A}(i,n)$ is the number of $A$-partitions of $A$. In the sequel, we will call the polynomial $f_{A,n}(t)$ as a $n$-th $A$-partition polynomial.

In a recent paper \cite{UZ} the case of $A$-partitions polynomials for $A=\{2^{i}:\;i\in\N\}$ were investigated. Among many properties of the $A$-partition polynomials, the authors proved that for $k=0, 1$, the expression $(-1)^{n}\sum_{i=0}^{n}(-1)^{i}i^{k}c_{A}(i,n)$ is non-negative. Moreover, the conjecture was stated that for each $k\in\N$ this expression is non-negative. This suggest the following general question.

\begin{ques}
Let $A\subset\N_{+}$ and $k\in\N$ be fixed. What can be said about the sign behaviour of the sequence
$$
S_{A,k}(n)=\sum_{i=0}^{n}(-1)^{i}i^{k}c_{A}(i,n)?
$$
In particular, under which conditions on $A$ we have that $(-1)^{n}S_{A,k}(n)\geq 0$ for given $k$ and all but finitely many $n$?
\end{ques}

Let us also note that the value $S_{A,0}(n)=\sum_{i=0}^{n}(-1)^{i}c_{A}(i,n)$ has a nice combinatorial interpretation as the difference between the number of partitions of $n$ with even number of parts and odd number of parts. One can think that in general the sign behaviour of $S_{A,0}(n)$ should be alternating, that means that $(-1)^{n}S_{A,0}(n)\geq 0$. The aim of this paper is to investigate this expectation and prove several results which in many cases allow the proof of the non-negativity of $(-1)^{n}S_{A,k}(n)$ for a broad class of sets $A$ and all $k, n\in\N$. However, as we will see this is not true in general. Our investigations fits in various studies concerning behaviour of various counting objects involving parts of partitions (see \cite{Mer}).

Let us describe the content of the paper in some details. In Section \ref{sec2} we collect some basic observations concerning the sequence $(S_{A,k}(n))_{n\in\N}$. In particular, by investigating so called $\delta$ operator we introduce two family of polynomials $p_{A,n}(t), q_{A,n}(t)$ closely related to the logarithmic derivative of the generating function of $A$-partition polynomials and obtain recurrence relations satisfied by terms of the double sequence $(S_{A,k}(n))_{k, n\in\N}$.
In Section \ref{sec3} we obtain several results which allow the proof of nonnegativity of $S_{A,k}(n)$ for each $k$ and $n$ and several sets $A$. In particular, if $A$ contains only odd numbers then $(-1)^{n}S_{A,k}(n)\geq 0$. The same statement is true for the set $A=\N_{+}$ and many others. Finally, in the last section we state several problems and conjectures which appeared during our investigations and we were unable to solve. We hope that this collection of problems will stimulate further activity in this area.

\section{Basic observations}\label{sec2}

In order to investigate the sign behaviour of the sequence $(S_{A,k}(n))_{n\in\N}$, where $k\in\N$ is fixed, we consider the differential operator $\delta:\;\R[t]\rightarrow \R[t]$ defined in the following way: for $f\in\R[t]$ we put
$$
\delta(f(t))=t\frac{df(t)}{dt}.
$$
For $k=0$ we put $\delta^{(0)}(f)=f$ and for $k\in\N_{+}$ we define by induction a $k$-th power of the operator $\delta$ as $\delta^{(k)}(f)$, where
$$
\delta^{(k)}(f)=\delta(\delta^{(k-1)}(f)).
$$
The basic properties of the operator $\delta$ can be summarized in the following well known

\begin{lem}\label{delta}
\begin{enumerate}
\item $\delta$ is linear operator, i.e., for all $u, v\in\R, f, g\in\R[t]$ we have $\delta(uf+vg)=u\delta(f)+v\delta(g)$;
\item if $f\in\R[t]\setminus\R$ then $\op{deg}\delta(f)=\op{deg}f$;
\item for $k\in\N$ we have
$$
\delta^{(k)}(fg)=\sum_{i=0}^{k}\binom{k}{i}\delta^{(i)}(f)\delta^{(k-i)}(g).
$$
\item if $f(t)=\sum_{i=0}^{n}a_{i}t^{i}\in\R[t]$ and $k\in\N$ is given then
$$
\delta^{(n)}(f(t))=\sum_{i=0}^{n}a_{i}i^{n}t^{i}
$$
\end{enumerate}
\end{lem}
\begin{proof}
The first two properties are easy consequences of the definition of the operator $\delta$. The proof of the third property is just the application of the Leibnitz rule together with induction on $k$. The last property follows from the linearity of the operator $\delta$ and the equality $\delta^{(k)}(t^{i})=i^{k}t^{i}$. We omit the simple details.
\end{proof}

As an immediate consequence of the fourth property of the operator $\delta$ we get the equality
$$
S_{A,k}(n)=\sum_{i=0}^{n}(-1)^{i}i^{k}c_{A}(i,n)=\delta^{(k)}(f_{A,n}(t))\mid_{t=-1}.
$$

In order to simplify the notation for a given $A\subset\N_{+}$ and $n\in\N_{+}$ we write
$$
A(n):=\{d\in A:\;d|n\},
$$
i.e., the set $A(n)$ contains those divisors of $n$ which are lying in $A$.

We prove certain identities which will be useful in our investigations and allow to express action of $\delta$ operator on $A$-partition polynomial in terms of $A$-partitions polynomials and a related sequence of polynomials. More precisely, we prove the following.

\begin{lem}\label{pqrelations}
Let $n\in\N$. Then
\begin{align*}
\sum_{i=1}^{n}p_{A,i}(t)f_{A,n-i}(t)&=tf_{A,n}'(t)=\delta(f_{A,n}(t)),\\
\sum_{i=1}^{n}q_{A,i}(t)f_{A,n-i}(t)&=nf_{A,n}(t),\\
tq'_{A,n}(t)&=np_{A,n}(t),
\end{align*}
where
$$
p_{A,n}(t)=\sum_{a\in A(n)}t^{n/a},
$$
and
$$
q_{A,n}(t)=\sum_{a\in A(n)}at^{n/a}.
$$
\end{lem}
\begin{proof}
Applying logarithmic differentiation with respect to the variable $t$, we get the following expansion
$$
\frac{t}{F_{A}(t,x)}\frac{\partial F_{A}(t,x)}{\partial t}=t\frac{\partial}{\partial t}(\log F_{A}(t,x))=\sum_{a\in A}\frac{tx^{a}}{1-tx^{a}}=\sum_{n=1}^{\infty}p_{A,n}(t)x^{n}.
$$
Thus, the first identity is a consequence of the equality
$$
t\frac{\partial F_{A}(t,x)}{\partial t}=\left(\sum_{n=1}^{\infty}p_{A,n}(t)x^{n}\right)F_{A}(t,x)=\sum_{n=0}^{\infty}\left(\sum_{i=1}^{n}p_{A,i}(n)f_{A,n-i}(t)\right)x^{n}
$$
by a comparison of coefficients near $x^{n}$ on far ends.

Similarly, applying logarithmic differentiation with respect to the variable $x$, we get
$$
\frac{x}{F_{A}(t,x)}\frac{\partial F_{A}(t,x)}{\partial x}=x\frac{\partial}{\partial x}(\log F_{A}(t,x))=\sum_{a\in A}\frac{atx^{a}}{1-tx^{a}}=\sum_{n=1}^{\infty}q_{A,n}(t)x^{n},
$$
and thus, form the equality
$$
x\frac{\partial F_{A}(t,x)}{\partial x}=\left(\sum_{n=1}^{\infty}q_{A,n}(t)x^{n}\right)F_{A}(t,x)=\sum_{n=0}^{\infty}\left(\sum_{i=1}^{n}q_{A,i}(n)f_{A,n-i}(t)\right)x^{n}
$$
and comparison of coefficients, we get the second identity form the statement.

The third identity is a direct consequence of the from of the polynomials $p_{A,n}, q_{A,n}$.

\end{proof}

\begin{rem}
{\rm Let us note that the polynomial $p_{A,n}(t)$ can be seen as a polynomial generalization of the number of $A$-divisors function $\tau_{A}(n)=\sum_{d\in A(n)}1=p_{A,n}(1)$. Similarly, the polynomial $q_{A,n}(t)$ can be seen as a polynomial analogue of the sum of $A$-divisors function $\sigma_{A}(n)=\sum_{d\in A(n)}d=q_{A,n}(1)$. }
\end{rem}

As an application of the above result we prove a double recurrence relation satisfied by the sequence $(S_{A,k}(n))_{k,n}$. More precisely, we have the following.

\begin{cor}\label{relationS}
Let $k\in\N$. We have the following identities
$$
S_{A,k+1}(n)=\sum_{i=1}^{n}\sum_{j=0}^{k}\binom{k}{j}\delta^{(j)}(p_{A,i}(t))\mid_{t=-1}S_{A,k-j}(n-i).
$$
and
$$
\delta^{(k)}(p_{A,n}(t))=\sum_{a\in A(n)}\left(\frac{n}{a}\right)^{k}t^{\frac{n}{a}}.
$$
\end{cor}
\begin{proof}
As a consequence of the first identity from Lemma \ref{pqrelations} we get the following chain of identities
\begin{align*}
S_{A,k+1}(n)&=\delta^{(k)}(\delta(f_{A,n}(t)))=\delta^{(k)}(tf_{A,n}'(t))\mid_{t=-1}\\
         &=\delta^{(k)}\left(\sum_{i=1}^{n}p_{A,i}(t)f_{A,n-i}(t)\right)\mid_{t=-1}\\
         &=\sum_{i=1}^{n}\delta^{(k)}(p_{A,i}(t)f_{A,n-i}(t))\mid_{t=-1}\\
         &=\sum_{i=1}^{n}\sum_{j=0}^{k}\binom{k}{j}\delta^{(j)}(p_{A,i}(t))\mid_{t=-1}\delta^{(k-j)}(f_{A,n-i}(t))\mid_{t=-1}\\
         &=\sum_{i=1}^{n}\sum_{j=0}^{k}\binom{k}{j}\delta^{(j)}(p_{A,i}(t))\mid_{t=-1}S_{A,k-j}(n-i)
\end{align*}
and hence the result.

The second identity from the statement is a simple consequence of the last property from Lemma \ref{delta}.

\end{proof}

\section{Results}\label{sec3}

In this section we prove several results which allow to prove non-negativity of $(-1)^{n}S_{A,k}(n)$ for a broad class of sets $A$. Moreover, we show that the sequence $(\op{sign}(S_{A,k}(n)))_{n\in\N}$, where $A=\{a,b\}, a<b, \gcd(a,b)=1$, is periodic of period $2a(b-a)$.

We start with the following result.

\begin{lem}\label{twosets}
Let $A_{1}, A_{2}\subset\N_{+}$ and suppose that $A_{1}\cap A_{2}=\emptyset$. If for each $k\in\N$ we have $(-1)^{n}S_{A_{i},k}(n)\geq 0$ for $i=1, 2$, then $(-1)^{n}S_{A_{1}\cup A_{2},k}(n)\geq 0$.
\end{lem}
\begin{proof}
We are interesting in the value of
$$
S_{A_{1}\cup A_{2},k}(n)=\delta^{(k)}(f_{A_{1}\cup A_{2},n}(t))\mid_{t=-1}.
$$
Because $A_{1}\cap A_{2}=\emptyset$ we have the identity $F_{A_{1}\cup A_{2}}(t,x)=F_{A_{1}}(t,x)F_{A_{2}}(t,x)$. As a consequence we deduce that
$$
f_{A_{1}\cup A_{2}, n}(t)=\sum_{i=0}^{n}f_{A_{1},i}(t)f_{A_{2},n-i}(t).
$$
Thus, from Lemma \ref{delta} we obtain
$$
\delta^{(k)}(f_{A_{1}\cup A_{2},n}(t))=\sum_{i=0}^{n}\sum_{j=0}^{k}\binom{k}{j}\delta^{(j)}(f_{A_{1},i}(t))\delta^{(k-j)}(f_{A_{2},n-i}(t)).
$$
Finally, we get
\begin{align*}
(-1)^{n}S_{A_{1}\cup A_{2},k}(n)&=(-1)^{n}\delta^{(k)}(f_{A_{1}\cup A_{2}, n}(t))\mid_{t=-1}\\
                        &=\sum_{i=0}^{n}\sum_{j=0}^{k}\binom{k}{j}\{(-1)^{i}S_{A_{1},j}(i)\}\{(-1)^{n-i}S_{A_{2},k-j}(n-i)\}\geq 0
\end{align*}
and from our assumptions we get the result.
\end{proof}

\begin{cor}\label{manysets}
Let $\cal{A}=\{A_{1}, \ldots, A_{m}\}$ be a possible infinite family, i.e., we allow $m=+\infty$, of pairwise disjoint subsets of $\N_{+}$ and suppose that for each $k, n\in\N$ we have $(-1)^{n}S_{A_{i},k}(n)\geq 0$ for $i=1, \ldots, m$. If $A=\cup_{i=1}^{m} A_{i}$, then for each $k, n\in\N$ we have $(-1)^{n}S_{A,k}(n)\geq 0$.
\end{cor}
\begin{proof}
If $m<+\infty$ then the statement follows from the simple induction on $m$ using Lemma \ref{twosets} with $A=A_{1}\cup\ldots\cup A_{m-1}, B=A_{m}$. We omit the simple details.

If $m=+\infty$ we proceed as follows. Let $n\in\N_{+}$. Because the family $\cal{A}$ is infinite and contains pairwise-disjoint sets there is an integer $s$ such that $n<\op{min}A_{i}$ for each $i>s$. Because for $(-1)^{n}S_{A_{i},k}(n)\geq 0$ for each $i\leq s$ the same is true for the set $\bigcup_{i=1}^{s}A_{i}$. Next, due to the fact that $n$ can not be represented as a sum of elements from the set $\bigcup_{i=m+1}^{\infty}A_{i}$ we get that $(-1)^{n}S_{B,k}(n)\geq 0$ - and hence the result.

\end{proof}

\begin{lem}\label{alterlem}
Let $A\subset 2\N+1$. Then, for each $k\in\N$ we have $(-1)^{n}S_{A,k}(n)\geq 0$.
\end{lem}
\begin{proof}
We proceed by induction on $k\in\N$. The generating function for the sequence $(S_{A,0}(n))_{n\in\N}$ is
$$
F_{A}(-1,x)=\sum_{n=0}^{\infty}S_{A,0}(n)x^{n}=\prod_{a\in A}\frac{1}{1+x^{a}}=\prod_{a\in A}\frac{1}{1-(-x)^{a}}=\sum_{n=0}^{\infty}(-1)^{n}p_{A}(n)x^{n},
$$
i.e., $S_{A, 0}(n)=(-1)^{n}p_{A}(n)$ and our statement holds for $k=0$. Let us suppose that our statement is true up to fixed $k\in\N$. We prove it for $k+1$. For $j\in\{0,\ldots, k\}$ we have that
$$
\delta^{(j)}(p_{A,n}(t))\mid_{t=-1}=\sum_{a\in A(n)}\left(\frac{n}{a}\right)^{j}(-1)^{\frac{n}{a}}=\begin{cases}\begin{array}{ll}
                                                                                                                -\sum_{a\in A(n)}\left(\frac{n}{a}\right)^{j}, & n\equiv 1\pmod*{2} \\
                                                                                                                \sum_{a\in A(n)}\left(\frac{n}{a}\right)^{j}, & n\equiv 0\pmod*{2}.
                                                                                                              \end{array}
\end{cases}
$$
Thus, for $n\in\N$, we have that $(-1)^{n}\delta^{(j)}(p_{A,n}(t))\mid_{t=-1}\geq 0$. To get the statement for $k+1$ we use Corollary \ref{relationS} and note that
\begin{align*}
(-1)^{n}S_{A,k+1}&(n)=(-1)^{n}\sum_{i=1}^{n}\sum_{j=0}^{k}\binom{k}{j}\delta^{(j)}(p_{A,i}(t))\mid_{t=-1}S_{A,k-j}(n-i)\\
                 &=\sum_{i=1}^{n}\sum_{j=0}^{k}\binom{k}{j}\{(-1)^{i}\delta^{(j)}(p_{A,i}(t))\mid_{t=-1}\}\{(-1)^{n-i}S_{A,k-j}(n-i)\}\geq 0.
\end{align*}
The first expression in curly parentheses is non-negative due to our observation above and the second expression in curly parentheses is non-negative from our induction hypothesis. Hence, the statement is true for $k+1$ and our theorem is proved.

\end{proof}

We generalize the above result and prove the following:

\begin{thm}\label{alterthm}
Let $A\subset 2\N+1, m\in\N_{+}$ and ${\bf p}=(p_{i})_{i\leq m}$ be an increasing possibly infinite sequence of non-negative integers with $p_{0}=0$, i.e., we allow $m=+\infty$. Let us put
$$
B=\bigcup_{j=0}^{m}2^{p_{j}}A.
$$
Then for each $n, k\in\N$ we have $(-1)^{n}S_{B,k}(n)\geq 0$.
\end{thm}
\begin{proof}
First let us note that
$$
B(n)=\{b\in B:\;b|n\}=\left\{b\in \bigcup_{j=0}^{m}2^{p_{j}}A:\; b|n\right\}=\bigcup_{j=0}^{m}A(n/2^{p_{j}}).
$$
Although we allow $m=+\infty$, for any given $n\in\N$, the sum on the right side is finite. Indeed, if $p_{j}>\nu_{2}(n)$ then $A(n/2^{p_{j}})=\emptyset$. To get the result we consider two cases: $m<+\infty$ and $m=+\infty$.

Suppose that $m<+\infty$. Let us define $q_{i}=p_{i+1}-p_{i}$. From our assumption on the sequence ${\bf p}$ we know that $q_{i}\geq 1$ for $i=0, \ldots, m-1$.

We note that the sign of $S_{B,0}(n)=f_{B,n}(-1)$ is the sign of the $n$-th coefficient of
\begin{align*}
\prod_{a\in B}\frac{1}{1+x^{a}}&=\prod_{a\in A}\prod_{j=0}^{m} \frac{1}{1+x^{2^{p_{j}}a}}=\prod_{a\in A}\prod_{j=0}^{m} \frac{1-x^{2^{p_{j}}a}}{1-x^{2^{p_{j}+1}a}}=\prod_{a\in A}\left(\frac{1}{1-x^{2^{p_{m}+1}a}}\right)F(x),\\
\end{align*}
where
\begin{align*}
F(x)=&\prod_{a\in A}\left(\left(1-x^{2^{p_{0}}a}\right)\left(\prod_{i=0}^{m-1}\frac{1-x^{2^{p_{i+1}}a}}{1-x^{2^{p_{i}+1}a}}\right)\right)\\
    &\prod_{a\in A}\left(\left(1-x^{2^{p_{0}}a}\right)\left(\prod_{i=0}^{m-1}\prod_{j=1}^{q_{i}-1}\left(1+x^{2^{p_{i}+j}a}\right)\right)\right)\\
    &\prod_{a\in A}\left(\left(1+(-x)^{2^{p_{0}}a}\right)\left(\prod_{i=0}^{m-1}\prod_{j=1}^{q_{i}-1}\left(1+(-x)^{2^{p_{i}+j}a}\right)\right)\right).
\end{align*}
Now let us note that $F(-x)$ is a generating function for the partition function, say $pd_{A}({\bf p},n)$, which counts partitions of $n$ into distinct parts from the set
$$
\bigcup_{i=0}^{k}\bigcup_{j=1}^{q_{i}-1}2^{p_{i}+j}A.
$$
Consequently we get that
\begin{align*}
    \prod_{a\in B}\frac{1}{1+x^{a}}&=\left(\sum_{n=0}^{\infty}p_{A}(n)x^{2^{p_{k}+1}a}\right)\left(\sum_{n=0}^{\infty}(-1)^{n}pd_{A}({\bf p},n)x^{n}\right)\\
    &=\sum_{n=0}^{\infty}\left(\sum_{i=0}^{\lfloor n/2^{p_{k}+1}\rfloor}(-1)^{n-2^{p_{k}+1}i}p_{A}(i)pd_{A}({\bf p},n-2^{p_{k}+1}i)\right)x^{n}\\
    &=\sum_{n=0}^{\infty}(-1)^{n}\left(\sum_{i=0}^{\lfloor n/2^{p_{k}+1}\rfloor}p_{A}(i)pd_{A}({\bf p},n-2^{p_{k}+1}i)\right)x^{n}
\end{align*}
and it is clear that $(-1)^{n}S_{B,0}(n)\geq 0$.

Next observe that for $n$ odd we have
\begin{align*}
\delta^{(i)}(p_{B,n}(-1))=\sum_{a\in B(n)}\left(\frac{n}{a}\right)^{i}(-1)^{\frac{n}{a}}.
\end{align*}
and this sum is smaller than zero because $(-1)^{n/a}=-1$. If $n$ is even, i.e.,. $n=2^{\nu_{2}(n)}m$, where $\nu_{2}(n)$ is the highest power of 2 which divides $n$, then we can group elements of $B_{2^{\nu_{2}(n)}m}$ into $M+1$-tuples $(2^{p_{0}}, a,2^{p_{1}}a,\ldots, 2^{p_{M}}a)$ with $a$ odd and
$$
M=\op{max}\{i\in\{1,\ldots, k\}:\;p_{i}\leq \nu_{2}(n)\}.
$$
We thus have
\begin{align*}
\delta^{(i)}(p_{B,n}(-1))&=\sum_{a\in A(m)}\left(\sum_{s=0}^{M}\left(\frac{n}{2^{p_{s}}a}\right)^{i}(-1)^{\frac{n}{2^{p_{s}}a}}\right).\\
\end{align*}
It is clear that for any given $n$, the only possible value of $s\in\{0,\ldots, M\}$ for which $\frac{n}{2^{p_{s}}a}$ may be an odd integer is $s=M$ and additionally the equality $p_{M}=\nu_{2}(n)$ need to be satisfied. We thus get that for each $i, n\in\N$ we have the inequality $(-1)^{n}\delta^{(i)}(p_{B,n}(t))\mid_{t=-1}\geq 0$.

We are ready to finish the proof that $(-1)^{n}S_{B,k}(n)\geq 0$. We proceed by induction on $k$. We already proved that the statement is true for $k=0$. Let us assume that it holds for $j=0, 1, \ldots, k$. Therefore,
\begin{align*}
(-1)^nS_{B,k+1}(n)&=\sum_{i=1}^{n}\sum_{j=0}^{k}\binom{k}{j}\{(-1)^{i}\delta^{(j)}(p_{B,i}(t))\mid_{t=-1}\}\{(-1)^{n-i}S_{B,k-j}(n-i)\}.
\end{align*}
As we observed, the first expression in the curly bracket is non-negative. The second is non-negative because of our induction assumption, and hence we get the result.

To get the result for $m=+\infty$ we need the following observation. For a given $A\subset \N_{+}$ and given $n\in\N$ we put
$$
U(n)=\{a\in A:\;a\leq n\}=A\cap [1,n].
$$
Let us observe that for a given $N\in\N_{+}$ and $n\leq N$ we have the identity
$$
f_{A,n}(t)=f_{U(N),n}(t).
$$
In consequence, for a given $k\in\N, n\leq N$, we have the equality $S_{A,k}(n)=S_{U(N),k}(n)$.

As a consequence of the above property we see that if $m=+\infty$ and $n$ is given, there is a $s\in\N_{+}$ such that $n<2^{p_{i}}$ for $i\geq s$. Indeed, the sequence ${\bf p}$ is increasing. Hence, if we put ${\bf p}'=(p_{i})_{i\leq s}$ then we have the equality
$$
f_{A,n}(t)=f_{U(2^{p_{s}}),n}(t)
$$
and hance we can apply our statement for the finite set ${\bf p}'$ and get the required result.

\end{proof}

The combination of Lemma \ref{alterlem} and Theorem \ref{alterthm} is a very useful tool which can be used to prove sign alternating property of the sequence $(S_{A,k}(n))_{n\in\N}$ for several sets $A$. We start with the following.

\begin{thm}\label{consecutive}
Let $m\in\N_{+}$ and put $A=\{1, 2, \ldots, m\}$. Then, for each $k, n\in\N$ we have $(-1)^{n}S_{A,k}(n)\geq 0$.
\end{thm}
\begin{proof}
To get the result it is enough to observe the equality of sets
$$
A=\{1,\ldots,m\}=\bigcup_{i=1}^{\lceil m/2\rceil}(2i-1)\{1,2,\ldots,2^{\lfloor\log_{2}(m/(2i-1))\rfloor}\}
$$
and apply Theorem \ref{alterthm} together with Lemma \ref{manysets}.
\end{proof}

\begin{cor}
Let $A=\N_{+}$. Then for each $k\in\N$ we have $(-1)^{n}S_{A,k}(n)\geq 0$.
\end{cor}
\begin{proof}
To get the result it is enough to note the equality
$$
\N_{+}=\bigcup_{i=0}^{\infty}(2\N+1)2^{i}
$$
and apply Theorem \ref{alterthm} with $p_{i}=i$ and $A=2\N+1$-the set of odd numbers.
\end{proof}

\begin{cor}
Let $m\in\N_{\geq 2}$ be odd or be a power of two. Let us put $B=\{m^{i}:\;i\in\N\}$. Then for each $k\in\N$ we have $(-1)^{n}S_{B,k}(n)\geq 0$.
\end{cor}
\begin{proof}
If $m$ is odd then we directly apply Lemma \ref{alterlem}. If $m=2^{s}$ for $s\in\N_{+}$ then we apply the decomposition into pairwise disjoint sets
$$
B=\bigcup_{j=0}^{\infty}\{1\}2^{js}.
$$
and apply Theorem \ref{alterthm} with $p_{j}=js$ for $j=1, 2\ldots$ and $A=\{1\}$.
\end{proof}

\begin{rem}
{\rm Let us note that if $m=2$ then we get that the \cite[Conjecture 5.1]{UZ} is true. }
\end{rem}

Our results shows that the alternating behaviour of the signs of the sums $S_{A,k}(n)$ is quite typical. This may suggest that for all sets $A$, each $k\in\N$ and sufficiently large $n\in\N$ we have $(-1)^{n}S_{A,k}(n)\geq 0$. This is clearly not the case. We prove the following.

\begin{thm}
Let $k\in\N, a,b\in\N_{+}$ and take $A=\{a,b\}$ with $a<b$ and $(a,b)=1$. If $k=0$ then the sequence $(\op{sign}(S_{A,k}(n)))_{n\in\N}$ is eventually periodic of period $2ab$.

If $k>0$ then the sequence $(\op{sign}(S_{A,k}(n)))_{n\in\N}$ is eventually periodic of period $2a(b-a)$.
\end{thm}
\begin{proof}
Let us recall that $S_{A,k}(n)=\sum_{i=0}^{n}i^{k}c_{A}(i,n)$, where $c_{A}(i,n)$ is the number of $A$-partitions of $n$ with exactly $i$-parts.  Because we have only two elements so we can calculate $c_A(i,n)$ directly as the number of non-negative solutions $(x,y)$ of the linear system of equations:
\begin{equation}\label{system}
\begin{cases}
\begin{array}{ll}
xa+yb&=n \\
x+y&=i
\end{array}
\end{cases}
\end{equation}
The necessary condition for solvability of (\ref{system}) is the condition $ia\leq n\leq ib$. Equivalently we have $\frac{n}{b}\leq i \leq \frac{n}{a}.$ Moreover, we have
\begin{align*}
n&=xa+yb=(i-y)a+yb=y(b-a)+ia,\\
y&=\frac{n-ia}{b-a}.
\end{align*}
Consequently, the solution exists if and only if $i\equiv n/a\equiv r(n) \mod (b-a)$ and $\frac{n}{b}\leq i \leq \frac{n}{a}$. Note that, the congruence $ai\equiv n\pmod*{b-a}$ has always a solutions under due to the fact that $\gcd(a,b-a)=1|n$. Thus, we get the formula
\begin{equation*}
c_A(i,n)=\begin{cases}\begin{array}{ll}
1 & \textrm{if $i\in [n/b,n/a]$ and $i\equiv r(n) \mod (b-a)$}\\
0 & \textrm{in the remaining cases}
\end{array}.
\end{cases}
\end{equation*}
From our reasoning we obtain the expression for $f_{A,n}(t)$, i.e.,
\begin{equation*}
f_{A,n}(t)=\sum_{\substack{j\in [n/b,n/a]\\
j\equiv r(n) \mod (b-a)}}t^j
\end{equation*}
and
\begin{equation*}
\delta^{(k)}(f_{A,n})(t)=\sum_{\substack{j\in [n/b,n/a]\\
j\equiv r(n) \mod (b-a)}}j^kt^j
\end{equation*}

First let's consider the case $k>0.$ Assume that $n$ is big enough such that the segment $[n/b,n/a]$ contains at least $3(b-a)$ elements.
Now substitute $-1$ for $t.$ If $b-a$ is even then the parity of $j$ is the same as the parity of $r(n)$ so every term $j^k(-1)^j$ has the same sign. In particular, the sign of the whole sum is the same as the sign of the last element in the sum. If $b-a$ is odd then we obtain an alternating sum where every next term is bigger then the previous one, so the sign of this sum will be also the sign of the last element in the sum. Hence
$$\op{sign}(\delta^{(k)}(f_{A,n})(-1))=\op{sign}(j^k(-1)^j)=\op{sign}((-1)^j),$$
where $j$ is the biggest number in the set $[n/b,n/a]$ satisfying the condition $j\equiv r(n)\pmod*{(b-a)}$.
Now if we add $a(b-a)$ to $n$ then $n/a$ will increase by $b-a$ and $r(n)$ will not change. So our maximal $j$ will be now our previous maximal $j$ plus $b-a$ and
from our reasoning it is clear that the period of our sum is a divisor of the number $2a(b-a).$

For $k=0$ looking at the last term is not sufficient, because if two divides $\#(\Z\cap [n/a,n/b])$ then our sum can be zero. But that is the only difference. As in $k>0$ case, if $b-a$ is even then the parity of our summand $j$ is the same and the sign of the whole sum is the same as the sign of the last element in the sum. If $b-a$ is odd and $2|\#(\Z\cap [n/a,n/b])$ then our sum is zero. Otherwise it has the same sign as the last element. Now, if we add $2ab$ to $n$ then the number $r(n)$, the parity of $\#(\Z\cap [n/a,n/b])$, and the parity of its last integer element stay the same. Thus we get that the period of our sign sequence is a divisor of the number $2ab$.
\end{proof}

\section{Some problems and conjectures}

In this section we collect some problems and conjectures which appeared during our investigations.

\begin{conj}
For any given $k\in\N$ and finite $A\subset \N_{+}$ with $\gcd(A)=1$, the sequence $(S_{A,k}(n))_{n\in\N}$ is eventually periodic.
\end{conj}

Let us consider the set $A=\{1, 2, 6\}$. One can show that the sequence of signs of $S_{A,k}(n)$ is eventually periodic with period 4 for $k=0, 1$. More precisely,
\begin{align*}
&(-1)^{\left\lfloor \frac{n+1}{2}\right\rfloor}S_{0,A}(n)>0\quad\mbox{for}\quad n\geq 3,\\
&(-1)^{\left\lfloor \frac{n+1}{2}\right\rfloor}S_{1,A}(n)>0\quad\mbox{for}\quad n\geq 8.
\end{align*}
This can be easily proved from the explicit expressions for the sequences $(S_{A,i}(n))_{n\in\N}$ for $i=0, 1$. Indeed, if $i=0$, then $S_{A,0}(n)=f_{A,n}(-1)$ is just the $n$-th coefficient in the power series expansion of the rational function $F_{A}(-1,x)=1/(1+x)(1+x^2)(1+x^6)$. Similarly, if $i=1$ we need to investigate coefficients of the power series expansion of the rational function
$$
\sum_{n=0}^{\infty}S_{A,1}(n)x^{n}=\partial_{t}F(t,x)\mid_{t=-1}=-\frac{x \left(3 x^6+2 x^5-x^4-x^3+x^2+x+1\right)}{(1+x)^2(1+x^2)(1+x^6)^2}.
$$
However, one can observe that $(-1)^{n}S_{A,k}(n)>0$ for $k\geq 2$ and $n\gg 0$. This and related numerical observations suggest the following

\begin{prob}
Let $A\subset \N_{+}$ and suppose that for some $k\in\N$  we have $(-1)^{n}S_{A,k}(n)\geq 0$ for $n \gg 0$. Then for each $i\in\N$ we have $(-1)^{n}S_{A,k+i}(n)\geq 0$ for $n \gg 0$.
\end{prob}

During our investigations we obtained several results concerning periodic behaviour of the sequence of signs of the sequence $(S_{A,k}(n))_{n\in\N}$. It is an interesting question whether we can find a set $A\subset \N_{+}$ such that the sign behaviour of the corresponding sequence (for some $k\in\N$) is not periodic.

\begin{conj}
Let $A=\{3n+1:\;n\in\N\}$. The sequence $(\op{sign}(f_{A,n}(-1)))_{n\in\N}$ is not periodic.
\end{conj}

We expect that the following is true

\begin{conj}
Let $m\in\N_{\geq 2}$ and take $U_{m}=\{i!:\;i=1, \ldots, m\}$. Then, for each $k\in\{0,\ldots, m-2\}$
$$
(-1)^{\left\lfloor \frac{n+1}{2}\right\rfloor}S_{U_{m},k}(n)\geq 0\quad\mbox{for}\quad n\gg 0,
$$
i.e., the sequence of signs of the sequence $(S_{U_{m},k}(n))_{n\in\N}$ is eventually periodic of period 4. Moreover, if $k\geq m-1$ then $(-1)^{n}S_{U_{m},k}(n)\geq 0$ for $n\gg 0$.
\end{conj}

If the above conjecture is true then clearly implies that if $A=\{i!:\;i\in\N_{+}\}$, then for each $k\in\N$, the sequence $(\op{sign}(S_{A,0}(n)))_{n\in\N}$ is eventually periodic of period 4.

\begin{prob}
Explain the behaviour of $(\op{sign}(S_{A,k}(n))_{n\in\N}$ for $A=\{m+1,\ldots, m+k\}$, where $k$ is fixed.
\end{prob}

Numerical calculations suggest that for $k\geq 3$ the behaviour is alternating.

One can also go one step further and note that $S_{A,k}(n)$ is just the value of the polynomial $\delta^{(k)}(f_{A,n}(t))$ at the simplest root of unity $t=-1$. This suggest to investigate the following general

\begin{prob}\label{rootsofunity}
Let $d\in\N_{\geq 3}$ and $\zeta_{d}$ be the $d$-th primitive root of unity. Let us write
$$
\delta^{(k)}(f_{A,n}(t))\mid_{t=\zeta_{d}}=\sum_{i=0}^{\phi(d)}u^{(d)}_{A,k}(i,n)\zeta_{d}^{i}.
$$
What can be said about the sign behaviour of the family of sequences $(u^{(d)}_{A,k}(i,n))_{n\in\N}$, where $k\in\N$ and  $i=0, 1, \ldots, n-1$ are fixed? Especially interesting is the case of $k=0$.
\end{prob}

Numerical calculations suggest that we can not expect periodic behaviour even in the basic case when $A=\N_{+}$. In fact, based on our calculations we expect that there are arbitrary long sequences of consecutive positive and negative values of $u^{(4)}_{A,0}(i,n)$ for $t=\zeta_{4}$, i.e., $\zeta_{4}^2+1=0$, $i=0, 1$. On the other hand, if $A=\{2^{i}:\;i\in\N\}$ then in \cite[Corollary 4.2]{UZ} the authors proved, among other things, that
$$
\sum_{i=0}^{4n+2}(-1)^{i}u^{(4)}_{A,0}(2i,8n+4)=\sum_{i=0}^{4n+1}(-1)^{i}u^{(4)}_{A,0}(2i+1,8n+4)=0,
$$
which is a consequence of the vanishing of the value at $t=\zeta_{4}$ of $8n+4$-th $A$-partition polynomial. Thus, a deeper study of Problem \ref{rootsofunity} may reveals interesting combinatorial identities.

\vskip 1cm

\noindent Maciej Ulas, Jagiellonian University, Faculty of Mathematics and Computer Science, Institute of
Mathematics, {\L}ojasiewicza 6, 30-348 Krak\'ow, Poland; email:
maciej.ulas@uj.edu.pl

\noindent Filip Gawron, Jagiellonian University, Faculty of Mathematics and Computer Science, Institute of
Mathematics, {\L}ojasiewicza 6, 30-348 Krak\'ow, Poland; email: filipux1@gmail.com

\end{document}